\documentclass[12pt,draft]{amsart}

\usepackage[all]{xy}
\usepackage{amsfonts}
\usepackage{amssymb}

\usepackage{enumerate}

\usepackage{color}

\newtheorem{teo}{Theorem}[section]
\newtheorem{lem}[teo]{Lemma}
\newtheorem{prop}[teo]{Proposition}

\theoremstyle{definition}
\newtheorem{dfn}[teo]{Definition}

\newtheorem{ex}[teo]{Example}

\def\<{\langle}
\def\>{\rangle}
\def\ss{\subset}
\def\sse{\subseteq}

\def\a{\alpha}

\def\g{\gamma}

\def\e{\varepsilon}

\def\t{\tau}

\def\f{{\varphi}}

\def\F{{\Phi}}

\def\C{{\mathbb C}}

\def\bK{{\mathbf{K}}}
\def\bL{{\mathbf{L}}}

\def\B{{\mathcal B}}
\def\A{{\mathcal A}}

\def\M{{\mathcal M}}
\def\cN{{\mathcal N}}

\def\1{\mathbf 1}

\begin{document}
\title[Compact operators on  Hilbert $C^*$-modules]
{Geometric essence of ``compact'' operators
on Hilbert $C^*$-modules}
\author{Evgenij Troitsky}
\thanks{This work is
supported by the Russian Foundation for Basic Research
under grant 16-01-00357.}
\address{Dept. of Mech. and Math., Moscow State University,
119991 GSP-1  Moscow, Russia}
\email{troitsky@mech.math.msu.su}
\urladdr{
http://mech.math.msu.su/\~{}troitsky}


\begin{abstract}
We introduce a uniform structure on any Hilbert $C^*$-module
$\cN$ and prove the following theorem: suppose,
$F:\M\to\cN$ is a bounded adjointable morphism of
Hilbert $C^*$-modules over $\A$ and $\cN$ is countably generated.
Then $F$ belongs to the Banach space generated by
operators $\theta_{x,y}$, $\theta_{x,y}(z):=x\<y,z\>$,
$x\in \cN$, $y,z\in\M$ (i.e. $F$ is $\A$-compact, or
``compact'')
if and only if $F$ maps the unit ball of $\M$
to a totally bounded set with respect to this uniform
structure (i.e. $F$ is a compact operator).
\end{abstract}

\maketitle

\section*{Introduction}
The equivalence of two definitions of compactness of
operators on a Hilbert space (to be approximated by finite-dimensional operators and to map bounded sets to totally bounded
sets) is an extremely useful source for the study of
compact operators.

That is why it is interesting to obtain some similar
equivalence in the case of Hilbert $C^*$-modules. 
For a long time this problem was considered as
not having a reasonable solution, because, roughly speaking,
$\A$-compact operators are far from $\C$-finite-dimensional
ones. 

Up to our knowledge,
the only attempt to obtain some results in this direction
was made very recently by D.~Ke\v{c}ki\'c and Z.~Lazovi\'c in
\cite{KeckicLazovic2018}. Namely, they have introduced
some system of pseudo-metrics 
(related to known topologies \cite{Pas1,Frank1990}, 
see also \cite{Magajna})
on the standard Hilbert $C^*$-module 
$\ell_2(\A)$ and the corresponding notion
of a totally bounded set, where $\A$ is a von Neumann algebra.
In other words, they have suggested to consider
total boundedness with respect to a uniform structure,
which is not induced by the norm of the Hilbert $C^*$-module
under consideration.
Unfortunately, their approach does not give a solution
of the above problem, because it works only for $W^*$-algebras
and for $\cN=\ell_2(\A)$ and they prove the equivalence
of $\A$-compactness (they name it ``compactness'') 
of an adjointable operator $F:\ell_2(\A)\to\ell_2(\A)$ 
and total boundedness
of $F(B)$, where $B$ is the unit ball of $\ell_2(\A)$, 
for $\A=\B(H)$,
the algebra of all operators on a Hilbert space $H$. Unfortunately, in general, even for commutative algebras, they are not equivalent,
but ``compactness'' implies compactness. 

Quite recently, Z.~Lazovi\'c
\cite{lazovic2018} has involved unital $C^*$-algebras
in the context, but the other above listed unsatisfactory moments remain.

Our approach (namely, a choice of some other system of 
pseudo-metrics to define a uniform structure) 
seems to be giving a solution to the problem
overcoming these difficulties, in particular, the
corresponding notion of total boundedness is defined
and has good properties: 
 for any $C^*$-algebra  (not only unital), 
 and for any Hilbert $C^*$-module over $\A$ (not only
 the standard one) and it gives
the equivalence of $\A$-compactness  
of an adjointable operator $F$ and total boundedness
of $F(B)$ for all $\M$ and $\cN$ (with the only restriction:
$\cN$ is supposed to be countably generated; this is a very
natural restriction by Lemma \ref{lem:obraz_a_comp}).
This equivalence is our main result 
(Theorem \ref{teo:mainteorem}).

\medskip
The paper is arranged in the following way. In Section \ref{sec:prelim}
we recall some facts and definitions from the theory
of $C^*$-algebras and Hilbert $C^*$-modules. Also
we give a technical definition of relatively $\A$-compact
operators (Definition \ref{dfn:AcompSubm}) and prove a
couple of properties of $\A$-compact operators to be used
later.

In Section \ref{sec:def_unif} we define our uniform
structure, prove some properties of the related total
boundedness, and formulate the main result.

In Section \ref{sec:case_N=A} we prove an important
particular case (more precisely, a variant of the
main result in a particular case) to be used in the   
proof of the general case.

In Section \ref{sec:gencasereduc} we prove the main
result using a circle of implications. Some necessary
facts are proved as separate lemmas.

\medskip
\textsc{Acknowledgment:} The author is indebted to
V.~M.~Manuilov for helpful discussions.

This work is
supported by the Russian Foundation for Basic Research
under grant 16-01-00357.

\section{Preliminaries}\label{sec:prelim}
We start with a couple of statements about states. The
first one is very well known \cite[Theorem 3.3.2]{Murphy}.
\begin{lem}\label{lem:comparkvadr}
For any state $\f$ on $\A$ and any $a\in\A$ one has
$|\f(a)|^2\le \f(a^*a)$.
\end{lem}

The following ``inverse'' statement is also
known, but we have not found an appropriate reference.

\begin{lem}\label{lem:estimfornonpositive}
For any $a\in \A$ there is a state $\f$ such that
$\|a\|\le 2 |\f(a)|.$  
\end{lem}

\begin{proof}
Decompose: $a=\frac{1}{2}(a+a^*) + i\cdot \frac{1}{2i} (a-a^*)$.
Then, by \cite[Theorem 3.3.6]{Murphy}, for some states $\f_1$ and $\f_2$,
$$
\|a\|\le \frac{1}{2} \|a+a^*\| +\frac{1}{2}\|i(a-a^*)\|=
\frac{1}{2} (\f_1(a+a^*)+\f_2(i(a-a^*))\le
$$
$$
\le \frac{1}{2} (|\f_1(a)+\overline{\f_1(a)}| +
|\f_2(a)-\overline{\f_2(a)}|)
\le 2 \sup_{\f\mbox{ is a state}} 
|\f(a)|.
$$
Since the set of states is *-weakly compact (\cite[Theorem 5.1.8]{Murphy}), 
the continuous function $\f\mapsto|\f(a)|$
reaches its maximum.
\end{proof}  

Now we will give some basic facts about Hilbert
$C^*$-modules over $\A$ and $\A$-compact operators. 
Details and proofs
can be found in books \cite{Lance,MTBook} and 
survey paper \cite{ManuilovTroit2000JMS}. Some other
directions joining Hilbert $C^*$-modules and operator
theory can be found in
 \cite{FMT2010Studia,PavlovTro2011,BlanchGogi,TroManAlg}.

\begin{dfn}
A (right) pre-Hilbert $C^*$-module over a $C^*$-algebra $\A$
is an $\A$-module equipped with an $\A$-\emph{inner product}
$\<.,.\>:\M\times\M\to \A$ being a sesquilinear form on the
underlying linear space and restricted to satisfy:
\begin{enumerate}
\item $\<x,x\> \ge 0$ for any $x\in\M$;
\item $\<x,x\> = 0$ if and only if $x=0$;
\item $\<y,x\>=\<x,y\>^*$ for any $x,y\in\M$;
\item $\<x,y\cdot a\>=\<x,y\>a$ for any $x,y\in\M$, $a\in\A$.
\end{enumerate}

A pre-Hilbert $C^*$-module over $\A$ is a 
\emph{Hilbert $C^*$-module}
if it is complete w.r.t. its norm $\|x\|=\|\<x,x\>\|^{1/2}$.

A Hilbert $C^*$-module $\M$ is \emph{countably generated}
if there exists a countable set of its elements with dense
set of $\A$-linear combinations.

We will denote by $\oplus$ the Hilbert sum of Hilbert
$C^*$-modules in an evident sense.
\end{dfn}

We have the following Cauchy-Schwartz inequality \cite{Pas1}
(see also \cite[Proposition 1.2.4]{MTBook}) for any $x,y\in\M$
\begin{equation}\label{eq:cau_schw}
\<x,y\>\<y,x\>\le \|y\|^2 \<x,x\>.
\end{equation}

\begin{dfn}\label{dfn:standard_hm}
The \emph{standard} Hilbert $C^*$-module $\ell_2(\A)$
(also denoted by $H_\A$) is the set of all infinite sequences
$a=(a_1,a_2,\dots,\infty)$, $a_i\in\A$, such that
the series $\sum_i (a_i)^* a_i$ is norm-convergent in $\A$.
It is equipped with the inner product $\<a,b\>=
\sum_i (a_i)^*b_i$, where $b=(b_1,b_2,\dots)$.

If $\A$ is unital, then $\ell_2(\A)$ is countably
generated.
\end{dfn}

One of the most nice properties of countably generated
modules is the following theorem \cite{Kasp} (see
\cite[Theorem 1.4.2]{MTBook}). We mean that an isomorphism
preserves the $C^*$-Hilbert structure.

\begin{teo}[Kasparov stabilization theorem]\label{teo:Kasp_st}
For any countably generated Hilbert $C^*$-module $\M$
over $\A$,
there exists an isomorphism of Hilbert $C^*$-modules
$\M\oplus \ell_2(\A)\cong \ell_2(\A)$.  
\end{teo}

\begin{dfn}\label{dfn:operators}
An \emph{operator} is a bounded $\A$-homomorphism.
An operator having an adjoint (in an evident sense) is 
\emph{adjointable} (see \cite[Section 2.1]{MTBook}).
We will denote the Banach space of all operators
$F: \M\to \cN$ by  $\bL(\M,\cN)$
and the Banach space of adjointable operators
by $\bL^*(\M,\cN)$. The space $\bL(\M,\M)$
is a Banach algebra and $\bL^*(\M,\M)$ is a $C^*$-algebra.
\end{dfn}

\begin{dfn}\label{dfn:Acompact}
An \emph{elementary} $\A$-\emph{compact} operator
$\theta_{x,y}:\M\to\cN$, where $x\in\cN$ 
and $y\in\M$, is defined as $\theta_{x,y}(z):=x\<y,z\>$.
Then the Banach space $\bK(\M,\cN)$
of $\A$-\emph{compact} operators
is the closure of the subspace generated by all 
elementary $\A$-compact operators in $\bL(\M,\cN)$.

Since $(\theta_{x,y})^*=\theta_{y,x}$, $\A$-compact operators
are adjointable.
\end{dfn}

Since $T\theta_{x,y}=\theta_{Tx,y}$ if $T\in \bL(\cN,\cN')$,
and $\theta_{x,y}S=\theta_{x,S^*y}$ if $S\in \bL(\M',\M)$,
we have the following statement (see also \cite[Section 2.2]{MTBook}).

\begin{prop}\label{prop:propAcomp}
The set of $\A$-compact operators on a Hilbert $C^*$-module 
is a (closed two-sided self-adjoint) ideal in the $C^*$-algebra
of adjointable endomorohisms.

If $F$ is an $\A$-compact operator $F:\M\to\cN$ and 
$F_1:\M\to\M$ and $F_2:\cN\to\cN$ are adjointable, then
$F\circ F_1$ and $F_2\circ F$ are $A$-compact.

If $F$ is an adjointable operator $F:\M\to\cN$ and
$K_1:\M\to\M$ and $K_2:\cN\to\cN$ are $\A$-compact, then
$F\circ K_1$ and $K_2\circ F$ are $A$-compact.
\end{prop}

\begin{lem}\label{lem:Acompdirsum}
Let $F:\M\to \cN_1\oplus \cN_2$ be an adjointable  
operator. Then $F$ is $\A$-compact if and only if
$p_1 F$ and $p_2 F$ are $\A$-compact, where 
$p_1: \cN_1\oplus \cN_2 \to \cN_1$ and
$p_2: \cN_1\oplus \cN_2 \to \cN_2$ are 
orthogonal projections.
\end{lem}

\begin{proof}
`Only if' follows immediately from Proposition \ref{prop:propAcomp}.

If $p_1 F$ is approximated by a combination of $\theta_{x_1,y_1}$,
$x_1\in\cN_1$, $y_1\in \M$, and 
$p_2 F$ is approximated by a combination of $\theta_{x_2,y_2}$,
$x_2\in\cN_1$, $y_2\in \M$, then $F=p_1^* p_1 F + p_2^* p_2 F$
is approximated by the sum of appropriate combinations of
$\theta_{p_1^*(x_1),y_1}$ and $\theta_{p_2^*(x_2),y_2}$.
\end{proof}

\begin{lem}\label{lem:obraz_a_comp}
Let $F:\M\to \cN$ be an $\A$-compact operator.
Then its image $F(\M)$
is contained in a countably generated module.
\end{lem}

\begin{proof}
Indeed, suppose $\e_n\longrightarrow 0$ as $n\to\infty$.
Then for each $n$ there exists a
combination of $\theta_{y,z}$, which is $\e_n$-close to $F$:
$$
\|\theta_{y(n,1),z(n,1)}(x)+\cdots+ \theta_{y(n,k(n)),z(n,k(n))}(x)-F(x)\|<\e_n\|x\|.
$$
The Hilbert $C^*$-module generated by all $y(i,j)$ contains $F(\M)$,
since, for any $x\in\M$ and any $\e>0$, taking $\e_n<\e/\|x\|$
we have
$$
\|\theta_{y(n,1),z(n,1)}(x)+\cdots+ \theta_{y(n,k(n)),z(n,k(n))}(x)-F(x)\|<\e_n\|x\|<\e.
$$
\end{proof}

For technical reasons (not for the formulation of the results)
we will need a modification of the above definition.

\begin{dfn}\label{dfn:AcompSubm}
We say that an operator $F:\M\to\cN$ is $\A$\emph{-compact
relatively a submodule} $\cN^0\ss\cN$ if $F(\M)\subset \cN^0$
and $F$ is $\A$-compact as an operator from $\M$ to $\cN^0$.
Roughly speaking this means that $y$ in $\theta_{y,z}$ can
be taken from $\cN^0$. Denote the set of these operators
by $\bK(\M,\cN;\cN^0)$.
\end{dfn}

Proposition \ref{prop:propAcomp} implies the following
lemma.

\begin{lem}\label{lem:rel_comp_compos}
Suppose $F\in \bK(\M,\cN;\cN^0)$, $G\in \bL(\cN,\cN_1)$
and $G(\cN^0)\subset \cN^0_1$. Then $GF\in \bK(\M,\cN_1;\cN^0_1)$.

Suppose $F\in \bK(\M,\cN;\cN^0)$ and $G\in \bL^*(\M_1,\M)$.
Then $FG\in \bK(\M_1,\cN;\cN^0)$.
\end{lem}


\section{Definitions and formulation of the main theorem}
\label{sec:def_unif}


Now we pass to the definition of the desired uniform structure
and totally bounded sets. 

\begin{dfn}\label{dfn:unif_str}
Any \emph{uniform structure} on
a non-empty space $X$ can be
defined by a system of \emph{pseudo-metrics} 
(see \cite[p.~188]{Kelley}),
i.e. functions $d_\a:X\times X \to [0,+\infty)$
restricted to satisfy:
\begin{enumerate}
\item[1)]
 the symmetry property
$d_\a(x,y)=d_\a(y,x)$;
\item[2)] the triangle inequality
$d_\a(x,z)\le d_\a(x,y) + d_\a(y,z)$;
\item[3)] the separation property: if $x\ne y$, then
$d_\a (x,y)>0$ for some $\a$.
\end{enumerate}
\end{dfn}

We will define \emph{totally bounded} sets directly
in our case in Definition \ref{dfn:totbaundset} below.

\begin{dfn}\label{dfn:admissyst}
Let $\cN$ be a Hilbert $C^*$-module over $\A$. A countable
system $X=\{x_{i}\}$ of its elements 
is called \emph{admissible} for a submodule $\cN_0\subset \cN$
(or $\cN^0$-admissible) if
 for each $x\in\cN^0$
partial sums of the series $\sum_i \<x,x_i\>\<x_i,x\>$ 
are bounded by $\<x,x\>$ 
and the series is convergent. 
In particular, $\|x_i\|\le 1$ for any $i$.
\end{dfn}

\begin{ex}\label{ex:firstex}
For the standard module $\ell_2(\A)$ over a unital
algebra $\A$ one can take for $X$ the natural base $\{e_i\}$.
In the case of $\ell_2(\A)$ over a general algebra $\A$,
one can take $x_i$ having only the $i$-th component nontrivial
and of norm $\le 1$.
The other important example is $X$ with only finitely many non-zero
elements in any module and an appropriate normalization.
\end{ex}

Denote by $\F$ a countable collection $\{\f_1,\f_2,\dots\}$
of states on $\A$. For each pair $(X,\F)$ 
with an $\cN^0$-admissible $X$,
consider the following pseudo-metrics
\begin{equation}\label{eq:defnmetr}
d_{X,\F}(x,y)^2:=\sup_k 
\sum_{i=k}^\infty |\f_k\left(\<x-y,x_i\>\right)|^2,\quad x,y\in \cN^0. 
\end{equation}

First, remark that this is a finite non-negative number.
Indeed, by Lemma \ref{lem:comparkvadr}
\begin{eqnarray*}
\sum_{i=k}^s|\f_k\left(\<x-y,x_i\>\right)|^2 & = &
\sum_{i=k}^s |\f_k\left(\<x_i,x-y\>\right)|^2
\le \f_k\left(\sum_{i=k}^s \<x-y,x_i\>\<x_i,x-y\>\right)\\
&\le &\left\|\sum_{i=k}^s \<x-y,x_i\>\<x_i,x-y\>\right\|
\le  \|x-y\|^2. 
\end{eqnarray*}
Since in (\ref{eq:defnmetr}) we have a series of non-negative
numbers, this estimation implies its convergence and the estimation
\begin{equation}\label{eq:sravnsobych}
d_{X,\F}(x,y)\le  \|x-y\|.
\end{equation}

For $x\ne y$ there exists $(X,\F)$ such that 
$d_{X,\F}(x,y)> \frac{1}{2} \|x-y\|$. Indeed, take $X$ with $x_1=\frac{x-y}
{\|x-y\|}$
and other $x_i=0$, and $\f_1$ such that 
$\f_1\left(\<x-y,x-y\>\right)>\frac{1}{2} \|x-y\|^2$.
Then for any $z$, inequality (\ref{eq:cau_schw}) implies
$$
\sum_i \<z,x_i\>\<x_i,z\>=\<z,x_1\>\<x_1,z\> \le \|x_1\|^2 \<z,z\>
\le \<z,z\>
$$
and
$$
d_{X,\F}(x,y)\ge |\f_1\left(\<x-y,x_1\>\right)|=
\frac{|\f_1\left(\<x-y,x-y\>\right)|}{\|x-y\|}
> \frac{1}{2} \|x-y\|.
$$

Let us verify the triangle inequality:
\begin{equation}\label{eq:triangle_inequa}
d_{X,\F}(u,v)\le d_{X,\F}(u,w) + d_{X,\F}(w,v), 
\end{equation}
which can be rewritten as
$$
d_{X,\F}(u-v,0)\le d_{X,\F}(u-w,0) + d_{X,\F}(w-v,0),
$$
and thus it is sufficient to prove
\begin{equation}\label{eq:triangle_inequa_1}
d_{X,\F}(z+x,0)\le d_{X,\F}(z,0) + d_{X,\F}(x,0).
\end{equation}
Take an arbitrary $\e>0$ and choose $k$ and $m$ such that
\begin{equation}\label{eq:triangle_inequa_3}
d_{X,\F}(z+x,0)<\sqrt{\sum_{i=k}^m |\f_i (\<z+x,x_i\>)|^2}
+\e. 
\end{equation}
We have
\begin{equation}\label{eq:triangle_inequa_4}
\sqrt{\sum_{i=k}^m |\f_i (\<z+x,x_i\>)|^2}\le 
\sqrt{\sum_{i=k}^m (|\f_i (\<z,x_i\>)|+|\f_i (\<x,x_i\>|)^2}
\end{equation}
Consider in $\C^{m-k+1}$ with the standard inner product
$\<.,.\>_\C$ and the norm $\|.\|_\C$ the following vectors with 
non-negative coordinates:
$$
\vec{a}=(a_1,\dots):=(|\f_k (\<z,x_k\>)|,\dots,|\f_m (\<z,x_m\>)|),
$$
$$
\vec{b}=(b_1,\dots):=(|\f_k (\<x,x_k\>)|,\dots,|\f_m (\<x,x_m\>)|).
$$
Then by the triangle inequality for $\|.\|_\C$,
$$
\sqrt{\sum_{i=k}^m (|\f_i (\<z,x_i\>)|+|\f_i (\<x,x_i\>|)^2}=
\sqrt{\sum_j (a_j+b_j)^2}=\|\vec{a}+\vec{b}\|_\C\le
\|\vec{a}\|_\C+\|\vec{b}\|_\C
$$
$$
=\sqrt{\sum_{i=k}^m |\f_i (\<z,x_i\>)|^2}
+\sqrt{\sum_{i=k}^m |\f_i (\<x,x_i\>)|^2}\le d_{X,\F}(z,0) + d_{X,\F}(x,0).
$$
Since $\e$ in (\ref{eq:triangle_inequa_3})  is arbitrary, together with (\ref{eq:triangle_inequa_4}) 
the last estimation gives (\ref{eq:triangle_inequa_1})  and hence 
(\ref{eq:triangle_inequa}).

Also, we have the following version of the triangle inequality:
\begin{equation}\label{eq:triangle_inequa_2}
d_{X,\F}(x+y,u+v)\le d_{X,\F}(x,u) + d_{X,\F}(y,v),
\end{equation}
since by (\ref{eq:triangle_inequa_1}),
$$
d_{X,\F}((x-u)+(y-v),0)\le d_{X,\F}(x-u,0) + d_{X,\F}(y-v,0).
$$

So, we have verifed that $d_{X,\F}$ satisfy the
conditions of Definition \ref{dfn:unif_str} 
and thus define a uniform structure on the unit ball of
$\cN^0$. 

\begin{dfn}\label{dfn:totbaundset}
A set $Y\ss \cN^0\ss \cN$ is \emph{totally bounded}
with respect to this uniform structure, if for any $(X,\F)$, 
where $X\ss \cN$ is $\cN^0$-admissible,
and any 
$\e>0$ there exists a finite collection $y_1,\dots,y_n$
of elements of $Y$ such that the sets
$$
\left\{ y\in Y\,|\, d_{X,\F}(y_i,y)<\e\right\}
$$  
form a cover of $Y$. This finite collection is an 
$\e$\emph{-net in $Y$ for} $d_{X,\F}$.

If so, we will say briefly that $Y$ is $(\cN,\cN^0)$-\emph{totally bounded}.
\end{dfn}

Now we are able to formulate our main result.

\begin{teo}[Main Theorem]\label{teo:mainteorem}
Suppose, $F:\M\to\cN$ is an adjointable operator and
$\cN$ is countably generated. Then $F$ is $\A$-compact
if and only if $F(B)$ is $(\cN,\cN)$-totally bounded,
where $B$ is the unit ball of $\M$.
\end{teo}

We will complete this section with the following property.

\begin{lem}\label{lem:directsum_totbu}
Suppose, the set $Y\subset \cN=\cN_1\oplus \cN_2$ 
is $(\cN,\cN^0)$-totally bounded, 
where $\cN^0$ is a countably
generated submodule. Then
$p_1 Y$ and $p_2 Y$ are 
$(\cN_1,\cN^0_1)$- and $(\cN_2,\cN^0_2)$-totally bounded, 
respectively, where $\cN^0_1=p_1(\cN^0)$ and
$\cN^0_2=p_2(\cN^0)$ are countably generated submodules,
and 
$p_1: \cN_1\oplus \cN_2 \to \cN_1$,
$p_2: \cN_1\oplus \cN_2 \to \cN_2$ are the orthogonal projections.

Conversely, if $p_1 Y$ and $p_2 Y$ are 
$(\cN_1,\cN^0_1)$- and $(\cN_2,\cN^0_2)$-totally bounded
for some $\cN^0_1$ and $\cN^0_2$, 
respectively, then $Y$ is $(\cN,\cN^0_1\oplus\cN^0_2)$-totally
bounded. Evidently $\cN^0_1\oplus\cN^0_2$ is countably generated
if $\cN^0_1$ and $\cN^0_2$ are countably generated.
\end{lem}

\begin{proof}
If $p_j Y$ is contained in a countably generated
submodule $\cN^0_j\ss \cN_j$, $j=1,2$, then $Y$
is contained in the countably generated module 
$\cN^0_1\oplus \cN^0_2$. Conversely, if $Y$ is
contained in a countably generated submodule $\cN^0$,
then $p_j \cN^0$ is countably generated and 
$p_j Y \subset p_j \cN^0$, $j=1,2$.

Denote by $J_j=p_j^*$ the corresponding inclusions
$J_j:\cN_j\hookrightarrow \cN_1\oplus \cN_2$, $j=1,2$.

Suppose $Y$  is $(\cN,\cN^0)$-totally bounded and $X=\{x_i\}$ is 
an admissible system for a countably generated submodule
$\cN^0_1\ss \cN_1$. Then $J_1X=\{J_1(x_i)\}$ is admissible for $\cN^0$ because
$$
\<x,J_1(x_i)\>\<J_1(x_i),x\>=\<p_1 x,x_i\>\<x_i,p_1 x\>.
$$
Let $y_1,\dots,y_s$ be an $\e$-net in $Y$ for $d_{J_1X,\F}$.
Then $p_1 y_1,\dots,p_1 y_s$ is an $\e$-net in $p_1 Y$ for 
$d_{X,\F}$.
Indeed, consider an arbitrary $z\in p_1 Y$. Then $z=p_1 y$ for
some $y\in Y$. Find $y_k$ such that $d_{J_1X,\F}(y,y_k)<\e$.
Then
\begin{eqnarray*}
d_{X,\F}^2(z,p_1 y_k)&=&\sup_k 
\sum_{i=k}^\infty |\f_k\left(\<z-p_1 y_k,x_i\>\right)|^2=
\sum_{i=k}^\infty |\f_k\left(\<p_1(y- y_k), x_i\>\right)|^2\\
&=&
\sum_{i=k}^\infty |\f_k\left(\<y- y_k, J_1( x_i)\>\right)|^2=
d_{J_1 X,\F}^2(y,y_k)<\e^2.
\end{eqnarray*}
Similarly for $j=2$.

Conversely, suppose that $p_j Y$ are 
$(\cN_j,\cN_j^0)$-totally bounded, $j=1,2$.
Let $X=\{x_i\}$ be an admissible system in $\cN$ for 
$\cN^0_1\oplus \cN^0_2$ and $\e>0$ is arbitrary.
Then $X_j:=\{p_j(x_i)\}$ is an admissible system in $\cN_j$
for $\cN^0_j$
and, for $u, v\in p_j Y$,
\begin{equation}\label{eq:projprojproj}
d_{X,\F} (J_j  u, J_j v) = 
d_{X_j,\F} ( u, v).
\end{equation}
Indeed, we obtain the convergence and can estimate the sum
using (as above) the equality
$$
\sum_{i=1}^s \<u-v,p_j x_i\>\<p_j x_i,u-v\>
= \sum_{i=1}^s \<J_j(u-v),x_i\>\< x_i,J_j(u-v)\>
$$
and, quite similarly, (\ref{eq:projprojproj}) follows from the equality
$$
\<J_j p_j u- J_j v,x_i\>= \<p_ju- v, p_j x_i\>,\qquad j=1,2.
$$

Suppose, $z_1,\dots,z_m$ is an $\e/4$-net 
in $p_1 Y$ for $d_{X_1,\F}$ and
$w_1,\dots,w_r$ is an $\e/4$-net 
in $p_2 Y$ for $d_{X_2,\F}$. 
Consider $\{z_k+w_s\}$, $k=1,\dots,m$,
$s=1,\dots,r$. Then $\{J_1 z_k+ J_2 w_s\}$ is an
$\e/2$-net in $p_1 Y \oplus p_2 Y$
 for $d_{X,\F}$.
Indeed, for any $J_1 p_1 y_1 +J_2 p_2 y_2$, $y_1,y_2\in Y$,
one can find $z_k$ and $w_s$
such that
$$
d_{X_1,\F}(p_1 y_1,z_k)<\e/4,\qquad
d_{X_2,\F}(p_2 y_2,w_s)<\e/4.
$$
Then by (\ref{eq:triangle_inequa_2})
and (\ref{eq:projprojproj})
\begin{eqnarray*}
d_{X,\F}(J_1 p_1 y_1 +J_2 p_2 y_2, J_1 z_k+ J_2 w_s)&\le& d_{X,\F}(J_1 p_1 y_1,J_1 z_k)+
d_{X,\F}(J_2 p_2 y_2, J_2 w_s) \\
&=&d_{X_1,\F}(p_1 y_1, z_k)+
d_{X_2,\F}(p_2 y_2,  w_s)<\e/2.
\end{eqnarray*}
Now find a subset $\{u_l\}\subset \{z_k+w_s\}$
formed by all elements of $\{z_k+w_s\}$ such that there exists an
element $u^*\in Y\sse p_1 Y \oplus p_2 Y$ 
with $d_{X,\F}(u^*,z_k+w_s)<\e/2$.  
Denote these $u^*$ by $u^*_l$, $l=1,\dots,L$. So,
\begin{enumerate}[1)]
\item for any $y\in Y$, there exists $l\in 1,\dots, L$
such that $d_{X,\F}(y,u_l)<\e/2$;
\item for each $l\in 1,\dots, L$, we have $d_{X,\F}(u^*_l,u_l)<\e/2$.
\end{enumerate}
By the triangle inequality, $\{u^*_l\}$ is a finite $\e$-net
in $Y$ for $d_{X,\F}$ and we are done.
\end{proof}

\section{The case $\cN\subset\A$}\label{sec:case_N=A}

We will prove the following statement in this section.

\begin{teo}\label{teo:mainfor_N=A}
Suppose that $G:\M\to \A$ is an adjointable
operator such that $G(\M)$ is contained in a countably
generated submodule $\cN^0\subset\A$.
Then 
\begin{enumerate}[\rm 1)]
\item if $G$ is $\A$-compact relatively $\cN^0$,
i.e. $G\in \bK(\M,\A;\cN^0)$, then
$G(B)$ is $(\A,\cN^0)$-totally bounded;
\item if $G(B)$ is $(\A,\cN^0)$-totally bounded, then $G$ is $\A$-compact, i.e. $G\in \bK(\M,\A;\A)=\bK(\M,\A)$.
\end{enumerate}
Here $B$ is the unit ball of $\M$ as above.
\end{teo}

We need the following statement.

\begin{lem}\label{lem:notacomp}
Let $F:\M\to \A$ be a bounded adjointable, but 
not an $\A$-compact operator.
Suppose, $K>0$ is a constant.
Then there exists $\delta>0$ such that
for any $z\in \A$ there exists an element $x\in \M$
with $\|x\|\le 1$ such that $\|z \alpha -F(x)\|>\delta$
for any $\alpha\in\A$ with $\|\alpha\|\le K$.
Taking $\a=0$ gives $\|F(x)\|>\delta$.
\end{lem}

\begin{proof}
Suppose the opposite: for any $\e>0$ there exists
an element $z\in \A$ such that for any $x$
of norm 1 there exists
$\a_x$, $\|\a_x\|\le K$,
$\|z \alpha_x -F(x)\|<\e$. 

Choose an element of an approximate unit $\omega$ of $\A$, 
such that $\|z-\omega z\|<\e$, $0\le \omega \le 1$. 
Then for any $x\in \M$ of norm 1,
\begin{eqnarray*}
\|\theta_{\omega^{1/2},\omega^{1/2}}(F (x)) - F(x)\|&\le&
\|\theta_{\omega^{1/2},\omega^{1/2}}(z \a_x) - F(x)\|
+\|\theta_{\omega^{1/2},\omega^{1/2}}(z \a_x - F(x))\|
+\e 
\\
&\le &
\|\omega z \a_x - F(x)\|+\e = \|z \a_x - F(x)\|+
\|(z-\omega z)\a_x \| +\e \\
&\le& \|z \a_x - F(x)\|+ K\e+ \e < (2+K) \e. 
\end{eqnarray*}
This means that $F$ can be approximated by $\A$-compact operators
$\theta_{\omega^{1/2},\omega^{1/2}}\circ F$
(see Proposition \ref{prop:propAcomp}). Hence, $F$ is $\A$-compact. A contradiction.
\end{proof}

\begin{proof}[Proof of Theorem \ref{teo:mainfor_N=A}]
Let $a_1,a_2,\dots$ be a countable system
of generators for $\cN^0$. We have $G(\M)\sse \cN^0$,
hence $\overline{G(\M)} \sse \cN^0$. 
Consider a separable $C^*$-subalgebra $\A_0\sse\A$ generated by these elements and its increasing countable approximate unit
$\omega_i$, such that $\omega_i\le 1$, $\omega_i\le \omega_j$, if
$j<j$, and
\begin{equation}\label{eq:propapprunit1}
\omega_j\omega_i=\omega_i,\qquad j\ge i,
\end{equation}
\begin{equation}\label{eq:propapprunit2}
(\omega_j-\omega_i)^2 \le \omega_j-\omega_i,\qquad j\ge i.
\end{equation}

Suppose, that $G(B)$ is $(\A,\cN^0)$-totally bounded,
but $G$ is not $\A$-compact.
Then, for $K=\|G\|$,
Lemma \ref{lem:notacomp} implies that 
for some $\delta>0$ and each $\omega_i$, 
there exists an element $z_i=G(x_i)$, 
with $\|x_i\|\le 1$, 
such that $\|z_i-\omega_i\beta\|>\delta$ for any 
$\beta\in\A$ with $\|\beta\|\le K$.
Also, $\|z_i\|>\delta$ (cf. the last line of the formulation
of Lemma \ref{lem:notacomp}). 
In particular, for $\beta=z_i$,
\begin{equation}\label{eq:elementoff}
|(1-\omega_i) z_i\|=\|z_i-\omega_i z_i\|>\delta
\end{equation}
(in the unitalization).
Now choose some sub-sequence $i(j)$ of $i$ in such a way that
\begin{equation}\label{eq:estimforzomega}
\|\omega_{i(j+1)} z_{i(j)} - z_{i(j)}\| < \delta/2. 
\end{equation}
This is possible to do, because one can approximate $z_{i(j)}$
with a finite linear combination $a_1\alpha_1+\cdots+a_N\a_N$, while
$\{\omega_i\}$ is an approximate unit for each of $a_i$.

From (\ref{eq:elementoff}) and (\ref{eq:estimforzomega})
we obtain the estimation:
\begin{equation}\label{eq:estimperedsosto}
\|(\omega_{i(j+1)}-\omega_{i(j)}) z_{i(j)}\|\ge
\|\omega_{i(j)} z_{i(j)}-z_{i(j)}\|-
\|z_{i(j)}-\omega_{i(j+1)} z_{i(j)}\|>\delta-\delta/2=\delta/2.
\end{equation}

Suppose that $G(B)$ is totally bounded and consider a
semi-norm $d_{X,\F}$ defined by $X=\{x_j\}
:=\{\omega_{i(j+1)}-\omega_{i(j)}\}$
and $\F=\{\f_1,\f_2,\dots\}$, where
\begin{equation}\label{eq:vyborphi}
|\f_j(\<z_{i(j)},x_j\>)|=|\f_j((z_{i(j)})^*x_j)|=
|\f_j(x_j z_{i(j)})|\ge \frac{\delta}{4}.
\end{equation}
This $X$ is admissible for $\cN^0$ because we can estimate the partial sums by (\ref{eq:propapprunit2}):
$$
\sum_{j=1}^s (\omega_{i(j+1)}-\omega_{i(j)})^2\le
\sum_{j=1}^s \omega_{i(j+1)}-\omega_{i(j)}=\omega_{i(s+1)} 
$$
and hence, for any $x\in \A$,
$$
\sum_{j=1}^s \<x,x_j\>\<x_j,x\>=x^* \: \left(
\sum_{j=1}^s (\omega_{i(j+1)}-\omega_{i(j)})^2\right)\: x
\le x^* \omega_{i(s+1)} x \le x^*x=\<x,x\>.
$$
Similarly, the convergence for $x\in \cN^0$ follows from the estimation 
$$
\sum_{j=k}^s \<x,x_j\>\<x_j,x\>=x^* \: \left(
\sum_{j=k}^s (\omega_{i(j+1)}-\omega_{i(j)})^2\right)\: x
\le x^* (\omega_{i(s+1)}-\omega_{i(k)}) x 
$$
$$
\le 
\|x\|\cdot \|(\omega_{i(s+1)}-\omega_{i(k)}) x\|
\le \|x\|\cdot \|x-\omega_{i(k)} x\|\to 0\qquad (k\to\infty).
$$
Also, these $\f_j$ satisfying (\ref{eq:vyborphi})
do exist by (\ref{eq:estimperedsosto}) and Lemma \ref{lem:estimfornonpositive}.

Then there exist $y_1, \dots, y_D$ in $G(B)$ such that
for any $y\in G(B)$ there exists $k\in \{1,\dots,D\}$ such
that $d_{X,\F}(y,y_k)<\delta/8$. One can find, as above (cf.
the argument after (\ref{eq:estimforzomega})) a number $j_0$
such that
$$
\|(1-\omega_{i(j)}) y_k\| < \delta/8, \qquad j\ge j_0,\quad
k=1,\dots,D,
$$
and hence, for $j\ge j_0$, $k=1,\dots,D$,
\begin{equation}\label{eq:finoffinsys}
\delta/8 > \|\omega_{i(j+1)} (1-\omega_{i(j)}) y_k\| = 
\|(\omega_{i(j+1)}-\omega_{i(j)}) y_k\| = 
\|\<x_{j},y_k\>\|=\|\<y_k,x_{j}\>\|.
 \end{equation} 
Then, for all $k=1,\dots,D$ and $y:=z_{i(j)}$,
$$
d_{X,\F}(y,y_k)\ge |\f_{j+1}(\<y-y_k,x_{j}\>)|
> |\f_{j}(\<z_{i(j)},x_{j}\>)|-\delta/8
\ge \delta/4-\delta/8=\delta/8.
$$
by (\ref{eq:finoffinsys}) and (\ref{eq:vyborphi}).
A contradiction with the choice of $y_1, \dots, y_D$
and the supposition that $G(B)$ is $(\A,\cN^0)$-totally bounded.
This proves Theorem \ref{teo:mainfor_N=A}
in one direction.

Now suppose that $G$ is an $\A$-compact operator
relatively $\cN^0$. Denote 
$$
c:=\max\{1,\|G\|\}.
$$

Consider $d_{X,\F}$ for some $X=\{x_i\}$ 
(admissible for $\cN^0$)
and
$\F=\{\f_i\}$.
Consider arbitrary small $\e>0$. We can suppose that $\e<1$.
Then, one can approximate $G$ with a finite
combination of $\theta_{u,v}$: for any $x\in\M$,
$$
\|\theta_{u(1),v(1)}(x)+\cdots+ \theta_{u(n),v(n)}(x)-G(x)\|<
\frac{\e}{55c^2}\|x\|,
$$
$$
\|u(1)\<v(1),x\>+\cdots+ u(n)\<v(n),x\>-G(x)\|<
\frac{\e}{55c^2}\|x\|,\qquad u(i)\in\cN^0.
$$
We can replace $u(i)$ by some arbitrary close $\A$-linear
combinations of generators $a_j$ of $\cN^0$
and obtain (for simplicity
of notation we take successively all $a_1$, ... $a_D$ for some $D$)
$$
\|a_1 \<w_1,x\>+\cdots + a_D \<w_D,x\>-G(x)\|<
\frac{\e}{54c^2}\|x\|.
$$
Now, for a sufficiently small $\tau >0$, namely, 
$\tau<  \frac{\e^2}{54^2 c^4\cdot D^2 \cdot(\sup_j \|w_j\|)^2}$,
consider $a:=a_1 (a_1)^*+\cdots +a_D (a_D)^*$
and $b:=a(\tau+a)^{-1}$, so $\|b\|\le 1$.
We have $a\in\cN^0$, and hence, $b\in\cN^0$.
For any $j=1,\dots,D$ we have
$$
(b-1)a_j a_j^*(b-1)\le (a(\tau+a)^{-1}-1) a (a(\tau+a)^{-1}-1) \le \frac{\tau}{2},
$$
because we have the following estimation for positive numbers $t$,
$$
\left(\frac{t}{\tau+t}-1\right)^2 t=\frac{\t^2 t}{(\t+t)^2}\le 
\frac{\t^2 t}{2\t t}=\frac{\t}{2}.
$$
Hence,
\begin{equation}
\|(b-1)a_j\|\le \sqrt{\t/2}\le \frac{\e}{54 c^2\cdot D \cdot\sup_j \|w_j\|}.
\end{equation}

Thus,
\begin{eqnarray}\label{eq:bbfozen}
\|G- b G\|&\le& \|\theta_{a_1,w_1}+\cdots +\theta_{a_D,w_D}-G\|
+  \|(\theta_{a_1,w_1}- b \theta_{a_1,w_1})+\cdots \\
&&+(\theta_{a_D,w_D}- b \theta_{a_D,w_D})\| 
+ \| b\|\cdot \|\theta_{a_1,w_1}+\cdots +\theta_{a_D,w_D}-G\|
\nonumber\\
&\le& \frac{\e}{54c^2} + D \cdot\sup_j (\|(1- b)a_j\|\cdot \|w_j\|)+
\frac{\e}{54c^2}
\le \frac{\e}{18c^2}. 
\nonumber
\end{eqnarray}
By the triangle inequality, (\ref{eq:sravnsobych}),
and (\ref{eq:bbfozen}), for $r,t\in B$, we have
\begin{eqnarray*}
d_{X,\F}(G(r),G(t))&\le& d_{X,\F}(bG(r),bG(t)) + d_{X,\F}(G(r),bG(r)) +d_{X,\F}(G(t),bG(t)) 
\\
&\le&d_{X,\F}(bG(r),bG(t)) + \frac{\e}{9c^2}
\end{eqnarray*} 
and
\begin{eqnarray}\label{eq:bbfozenquadr}
d_{X,\F}(G(r),G(t))^2&\le& d_{X,\F}(bG(r),bG(t))^2 + \frac{\e^2}{81c^4}
+ 2 \frac{\e}{9c^2}\cdot \|G\|^2\\
&\le & d_{X,\F}(bG(r),bG(t))^2 + \frac{\e}{81}+ \frac{2\e}{9}
\nonumber\\
&\le & d_{X,\F}(bG(r),bG(t))^2 + \frac{\e}{3}.
\nonumber
\end{eqnarray} 

Since $b\in\cN^0$, the series
$
\sum_i \<b,x_i\>\<x_i,b\>
$
is convergent by the definition, and we can find 
a sufficiently large $K$ such that
$$
\left\| \sum_{i=K+1}^\infty \<b,x_i\>\<x_i,b\>\right\|<
\frac{\e}{12\|G\|^2}.
$$ 
Thus for any $x\in \M$,
\begin{eqnarray}\label{eq:estimhvostomega}
\left\| \sum_{i=K+1}^\infty \<b  G(x),x_i\>\<x_i,bG(x)\>\right\|
&=& \left\| 
(G(x))^*\left(\sum_{i=K+1}^\infty \<b ,x_i\>\<x_i,b\>
\right) G(x)\right\| \\
&\le &\|G(x)\|^2 \left\| \sum_{i=K+1}^\infty \<b,x_i\>\<x_i,b\>\right\|
< \frac{\e}{12}\|x\|^2.\nonumber
 \end{eqnarray} 
Taking into the account Lemma \ref{lem:comparkvadr} 
and (\ref{eq:estimhvostomega})
we have for $k> K$,
\begin{equation}\label{eq:estim_hvost}
\sum_{i=k}^\infty |\f_k\left(\<bG(r-t),x_i\>\right)|^2
\le 
\f_k\:\left(
\sum_{i=k}^\infty \<bG(r-t),x_i\>
\<x_i,bG(r-t)\>\right)\le \frac{\e}{12} 2^2=\frac{\e}{3}.
\end{equation}
Thus, by (\ref{eq:estim_hvost}) and (\ref{eq:bbfozenquadr}),
\begin{equation}\label{eq:ozenka_posled}
d_{X,\F}(G(r),G(t))^2\le \sup_{k\le K} 
\sum_{i=k}^K |\f_k\left(\<bG(r-t),x_i\>\right)|^2+\frac{\e}{3}+ \frac{\e}{3}.
\end{equation}
The first summand can be considered as the calculation of
the distance in the finite-di\-men\-si\-on\-al complex space $\C^d$
of dimension $d=K+(K-1)+...+1$ with the norm
\begin{equation}\label{eq:norm_comple}
\|(u_{1,1},\dots,u_{1,K}, u_{2,1},\dots, u_{2,K-1},\dots, u_{K,K})\|^2
\qquad\qquad\qquad\qquad\qquad
\end{equation}
$$
\qquad
\qquad\qquad=\sup\left\{|u_{1,1}|^2+\cdots+|u_{1,K}|^2,
|u_{2,1}|^2+\cdots+|u_{2,K-1}|^2,\dots, |u_{K,K}|^2
\right\}
$$
between the images of $G(r)$ and $G(t)$ under the following bounded
$\C$-linear map 
$$
R: G(B)\to \C^d,\qquad y \mapsto \{\f_k\left(\<x_i,b y\>\right)\},\qquad
k=1,\dots,K,\quad i=k,\dots,K.
$$ 
(We have transposed entries to have not an anti-linear, but a
linear map.)
Thus one can find an $\e/3$-net in $G(B)$ for the semi-norm
$\g$, being the composition of $R$ and (\ref{eq:norm_comple}).
Namely, find an $\e/3$-net $R(G(t_1)),\dots, R(G(t_s))$,
$t_i\in B$,
in $\C^d$ for the bounded set
$R(G(B))$ and the norm (\ref{eq:norm_comple}). 
Then  $G(t_1),\dots, G(t_s)$ will be
an $\e/3$-net for $\g$.
By (\ref{eq:ozenka_posled}) $G(t_1),\dots, G(t_s)$
 will be an $\e$-net for $d_{X,\F}$.
\end{proof}

\section{The general case}\label{sec:gencasereduc}
\begin{dfn}
Denote by $q_n$ the orthogonal projection
in $\ell_2(\A)$ onto the $n$-th summand of the standard
decomposition $\ell_2(\A)$.

Denote by $Q_n:=q_1+\dots+q_n$ the orthogonal projection
onto the module $L_n\cong \A^n$ formed by the first $n$ 
standard summands.
\end{dfn}

The proof of the general case of Main Theorem \ref{teo:mainteorem}
will be done in the following way and uses, in particular, a 
reduction to Theorem \ref{teo:mainfor_N=A}.

\begin{proof}[Proof of Theorem \ref{teo:mainteorem}]
Denote by 
$S$ the Kasparov stabilization $S:\cN\to \cN\oplus \ell_2(\A)\cong
\ell_2(\A)$ and prove the theorem moving along the following 
cycle of statements:

\noindent
\fbox{1} $F:\M\to \cN$ is an $\A$-compact operator and $\cN$
is a countably generated $\A$-module.

\begin{quote}
If a combination of $\theta_{x,y}$ approximates $F$,
where $x\in\cN$, $y\in \M$, then the same combination of
$\theta_{S(x),y}$ approximates $F'=S\circ F$. Thus, we obtain:
\end{quote}

\noindent
\fbox{2} The image of its Kasparov stabilization $F':\M\to \cN\ss \cN\oplus \ell_2(\A)$  is contained in a
countably generated module $S(\cN)$ and $F'$
is $\A$-compact relatively $S(\cN)$. 

\begin{quote}
By Lemma \ref{lem:epsilonapprell2_acomp_back} below this implies:
\end{quote}

\noindent
\fbox{3} For arbitrary $\e>0$ there exists $D$ such that
$\|F'-Q_D F'\|<\e$, the image of $Q_D F'$ is contained
in a countably generated module $Q_D S(\cN)$  and
$Q_D F'$ is $\A$-compact relatively $Q_D S(\cN)$.

\begin{quote}
By Lemma \ref{lem:rel_comp_compos}
 we obtain:
\end{quote}

\noindent
\fbox{4} For arbitrary $\e>0$ there exists $D$ such that
$\|F'-Q_D F'\|<\e$, the image of $q_i F'$ is contained
in a countably generated module $q_i S(\cN)$  and
$q_i F'$ is $\A$-compact relatively  $q_i S(\cN)$, $i=1,\dots,D$.

\begin{quote}
By Theorem \ref{teo:mainfor_N=A} we arrive to:
\end{quote}

\noindent
\fbox{5} For arbitrary $\e>0$ there exists $D$ such that
$\|F'-Q_D F'\|<\e$, the image of $q_i F'$ is contained
in a countably generated module $q_i S(\cN)$  and
$q_i F'(B)\ss\A$ is $(\A,q_i S(\cN))$-totally bounded, 
$i=1,\dots,D$.

\begin{quote}
Apply inductively Lemma \ref{lem:directsum_totbu},
keeping in mind that $q_i Q_D=q_i$ and obtain:
\end{quote}

\noindent
\fbox{6} For arbitrary $\e>0$ there exists $D$ such that
$\|F'-Q_D F'\|<\e$, the image of $Q_D F'$ is contained
in a countably generated module $q_1 S(\cN)\oplus\cdots\oplus
q_D S(\cN)$  and
$Q_D F'(B)\sse\A\oplus\cdots\oplus\A$ is 
$(\A\oplus\cdots\oplus\A,q_1 S(\cN)\oplus\cdots\oplus
q_D S(\cN))$-totally bounded.

\begin{quote}
By Lemma \ref{lem:epsilonapprell2_totbun} below we have:
\end{quote}

\noindent
\fbox{7}  The image of $F':\M\to \ell_2(\A)$ is contained
in a countably generated module $\cN^1:=q_1 S(\cN)\oplus\cdots\oplus
q_n S(\cN)\oplus \cdots$ (where the $C^*$-Hilbert sum supposes
taking the closure) and
$F'(B)$ is 
$(\ell_2(\A),\cN^1)$-totally bounded.

\begin{quote}
We have $S^*(\cN^1)=\cN$. Indeed, 
for an arbitrary small $\e$
and any $x\in S(\cN)$ we can find $Q_n$ such that
$\|Q_n x-x\|<\e$, where we have $Q_n(x)=q_1 x +\cdots +q_n x \in\cN^1$.
Since $\cN^1$ is closed, this implies that $S$
maps injectively $\cN \to \cN^1$. Then $S(\cN)$ has an
orthogonal complement $\cN^2$ in $\cN^1$, namely,
$\cN^2=\cN^1\cap (S(\cN))^\bot$, where $(S(\cN))^\bot$
is the orthogonal complement in $\ell_2(\A)$.
Thus, we have $\cN^1=S(\cN)\oplus \cN^2$.
Since $S^*S(y)=y$ and $S^*(z)=0$ if $z\in \cN^2$,
$S^*:\cN^1\to \cN$ is a surjection. 
Now we can apply Lemma \ref{lem:directsum_totbu} 
to the direct sum $\cN^1=S(\cN)\oplus \cN^2$ to obtain:
\end{quote}

\noindent
\fbox{8} $F(B)$ is 
$(\cN,\cN)$-totally bounded.

\begin{quote}
We can apply Lemma \ref{lem:directsum_totbu} 
``in the opposite
direction'', because evidently $(1-P)F'(B)=0$ is totally bounded,
where $P=SS^*$ is the orthogonal projection onto $S(\cN)$. Then we 
have:
\end{quote}

\noindent
\fbox{9}  
$F'(B)$ is 
$(\ell_2(\A),S(\cN))$-totally bounded.

\begin{quote}
By Lemma \ref{lem:epsilonapprell2_totbun} below
we obtain:
\end{quote}

\noindent
\fbox{10} For arbitrary $\e>0$ there exists $D$ such that
$\|F'-Q_D F'\|<\e$, the image of $Q_D F'$ is contained
in a countably generated module $Q_D S(\cN)$  and
$Q_D F'(B)\sse\A\oplus\cdots\oplus\A$ is 
$(\A\oplus\cdots\oplus\A,Q_D S(\cN))$-totally bounded.

\begin{quote}
Apply inductively Lemma \ref{lem:directsum_totbu}
(keeping in mind that $q_i Q_D=q_i$) to obtain:
\end{quote}

\noindent
\fbox{11}  
For arbitrary $\e>0$ there exists $D$ such that
$\|F'-Q_D F'\|<\e$, the image of $q_i F'$ is contained
in a countably generated module $q_i S(\cN)$  and
$q_i F'(B)\ss\A$ is 
$(\A,p_i S(\cN))$-totally bounded, $i=1,\dots,D$.

\begin{quote}
By Theorem \ref{teo:mainfor_N=A} we obtain:
\end{quote}

\noindent
\fbox{12} For arbitrary $\e>0$ there exists $D$ such that
$\|F'-Q_D F'\|<\e$, the image of $q_i F'$ is contained
in a countably generated module $q_i S(\cN)$  and
$q_i F':\M\to\A$ is $\A$-compact, $i=1,\dots,D$.

\begin{quote}
Applying inductively Lemma \ref{lem:Acompdirsum}
we arrive to:
\end{quote}

\noindent
\fbox{13} For arbitrary $\e>0$ there exists $D$ such that
$\|F'-Q_D F'\|<\e$ and
$Q_D F':\M\to L_D$ is $\A$-compact.

\begin{quote}
The operator $F'$ is approximated by $\A$-compact
operators $Q_D F'$. Thus, $F'$ is $\A$-compact (find more detail
in Lemma \ref{lem:epsilonapprell2_acomp}):
\end{quote}

\noindent
\fbox{14}  
$F'$ is $\A$-compact. 

\begin{quote}
 Since $F=S^*\circ F'$, Proposition \ref{prop:propAcomp}
 implies:
\end{quote}

\noindent
\fbox{15} $F:\M\to\cN$ is $\A$ compact.
\end{proof}

\begin{lem}\label{lem:epsilonapprell2_acomp}
Suppose $F':\M\to \ell_2(\A)$ is an adjointable operator and
for any $\e>0$
there exists an integer $D$ such that
\begin{enumerate}[\rm(1)]
\item $\|F'-Q_D F'\|<\e$,
\item $ Q_D F' $ is an $\A$-compact operator (as an operator
$\ell_2(\A)\to L_D$).
   \end{enumerate} 
Then $F'$ is $\A$-compact.    
\end{lem}

\begin{proof}
Since $L_D$ has an orthogonal complement in $\ell_2(\A)$,
the operator $ Q_D F' $ is $\A$-compact as an operator
$\ell_2(\A)\to\ell_2(\A)$. Hence, $F'$ is approximated by
$\A$-compact operators and $F'$ is $\A$-compact itself.
\end{proof}

\begin{lem}\label{lem:epsilonapprell2_acomp_back}
Suppose the image of an adjointable operator
$F':\M\to \ell_2(\A)$ is contained in a countably
generated module $\cN^0$ and $F'$ is an $\A$-compact operator
relatively $\cN^0$.
Then for any $\e>0$
there exists an integer $D$ such that
\begin{enumerate}[\rm(1)]
\item $\|F'-Q_D F'\|<\e$,
\item $ Q_k F' $ is an $\A$-compact operator relatively
$Q_k \cN^0$ for any $k$.
   \end{enumerate} 
\end{lem}

\begin{proof}
Approximate $F'$:
$$
\|F'-(\theta_{x_1,y_1}+\dots +\theta_{x_r,y_r})\|<\e/3,\qquad x_i,y_i\in \cN^0.
$$
Find a sufficiently large $D$ such that
$$
\|x_j-Q_D x_j\|<\frac{\e}{3 r \|y_j\|},\qquad j=1,\dots,r.
$$
Then
$$
\|\theta_{x_j,y_j}-Q_D \theta_{x_j,y_j}\|<\frac{\e}{3 r},\qquad j=1,\dots,r,
$$
and
\begin{eqnarray*}
\|F'-Q_D F'\|&\le & \|F'-(\theta_{x_1,y_1}+\dots +\theta_{x_r,y_r})\|
+ \|Q_D(F'-(\theta_{x_1,y_1}+\dots +\theta_{x_r,y_r}))\|\\
&& + \sum_{j=1}^r \|\theta_{x_j,y_j}-Q_D \theta_{x_j,y_j}\|
<\frac{\e}{3} + \frac{\e}{3} + r \frac{\e}{3r}=\e.
\end{eqnarray*}

Once again use the first inequality of the proof and obtain
for an arbitrary $k$ and $\e/3$:
$$
\e/3>\|Q_k (F'-(\theta_{x_1,y_1}+\dots +\theta_{x_r,y_r})\|=
\|Q_k F'-(\theta_{Q_k(x_1),y_1}+\dots +\theta_{Q_k(x_r),y_r})\|.
$$
Since $Q_k(x_j)\in Q_k(\cN^0)$, we are done.
\end{proof}

\begin{lem}\label{lem:epsilonapprell2_totbun}
Suppose $F':\M\to \ell_2(\A)$ is an adjointable operator.
If $F'(B)$ is totally bounded relatively
a countably generated submodule $\cN^0$ if and only if 
\begin{enumerate}[\rm 1)]
\item for any $\e>0$
there exists an integer $D$ such that
$\|F'-Q_D F'\|<\e$,
\item $ Q_k F'(\M)$ is contained in a countably generated 
submodule $\cN^0_k$, satisfying 
\begin{equation}\label{eq:coher}
Q_k \cN^0_{k+1}=\cN^0_k
\end{equation}
for any $k$.
\item $ Q_k F' (B)$ is totally bounded 
relatively $\cN^0_k$ for any $k$,
   \end{enumerate}   
   
Conversely, if 1), 2), and 3) have place, then
$F'$ is totally bounded relatively the countably generated
submodule  $\cN^0:=\overline{+_k \cN^0_k}$ in $\ell_2(\A)$. 
\end{lem}

\begin{proof}
Suppose that $F'(B)$ is totally bounded  relatively
a countably generated submodule $\cN^0$.
Take $\cN^0_k$ to be $Q_k \cN^0$.
Then evidently, 2) takes place, while
3) follow from Lemma \ref{lem:directsum_totbu}
for mutually complement projections $Q_k$ and $1-Q_k$.
Suppose that 1) is not true. Then there exists $\delta>0$
and a sequence of elements $z_i\in B$ and a corresponding
increasing sequence of numbers $j(i)\to\infty$ such that 
$\|q_{j(i)} F'(z_i)\|>\delta$. Let $\mu_i$ be an element
of an approximate unit of $\A$ (generally uncountable) such that
\begin{equation}\label{eq:ozenkasomegojj}
\| \mu_i q_{j(i)} F'(z_i)\|>\frac{3}{4}\delta. 
 \end{equation} 
 Suppose $x_i\in \ell_2(\A)$ has $\mu_i$ at the $j(i)$-th place
 and zeros on the remaining ones. In other words,
 $$
q_{j(i)} x_i =  \mu_i, \qquad (1-q_{j(i)}) x_i =0.
 $$
Rewrite (\ref{eq:ozenkasomegojj}) as
$$
\left\| \mu_i q_{j(i)} \frac{F'(z_i)}{\|F'\|}\right\|>
\frac{3\delta}{4\,\|F'\|} 
$$ 
and use Lemma \ref{lem:estimfornonpositive}
to find a state $\f_i$ such that
$$
\left| \f_i\left(\mu_i q_{j(i)} \frac{F'(z_i)}{\|F'\|}\right)\right|>
\frac{1}{2}\cdot \frac{3\delta}{4\,\|F'\|} ,
$$
or
\begin{equation}\label{eq:bolshiedalekieell}
|\f_i (\mu_i q_{j(i)} F'(z_i))|>\frac{3\delta}{8\,\|F'\|}.
\end{equation}
For this data there exists
$d_{X,\F}$ with $X=\{x_i\}$ and $\F=\{\f_i\}$, because
$X$ is evidently admissible even for the entire $\ell_2(\A)$
(cf. Example \ref{ex:firstex}). 
Thus one can find a
finite collection $\{y_1,\dots,y_n\}$ of elements of $B$
such that for any $y\in B$ there exists $k\in \{1,\dots,n\}$
with
\begin{equation}\label{eq:ozenkadliasetiell}
d_{X,\F}(F'(y),F'(y_k))<\frac{\delta}{8\,\|F'\|}.
\end{equation}
There exists a sufficiently large $D$ such that
\begin{equation}\label{eq:ubyvaniesetiell}
\|(1-Q_D) F'(y_k)\|<\frac{\delta}{8\,\|F'\|},\qquad k=1,\dots,n.
\end{equation}
Then for $j(i)>D$ and any $k\in \{1,\dots,n\}$, 
by (\ref{eq:bolshiedalekieell}) and (\ref{eq:ubyvaniesetiell})
we have
$$
d_{X,\F} (F'(z_i),F'(y_k))
\ge |\f_i(\<F'(z_i)-F'(y_k),x_i\>)|=
|\f_i(\<x_i,F'(z_i)-F'(y_k)\>)|
$$
$$
=|\f_i(\mu_i (q_{j(i)} F'(z_i))) -\f_i(\mu_i (q_{j(i)} F'(y_k)))|
\ge \frac{3\delta}{8 \|F'\|} - \frac{\delta}{8\,\|F'\|} = \frac{\delta}{4\,\|F'\|}.
$$
A contradiction with (\ref{eq:ozenkadliasetiell}).

Conversely, 
items 1) and 2) imply that $F'(\M)$ is contained in
the above defined module $\cN^0$, which is
a countably generated submodule with a set of generators being
a union of countable generating sets for all $Q_k F'(\M)$
(cf. the proof of Lemma \ref{lem:obraz_a_comp}). 

Let $d_{X,\F}$ be a seminorm for some admissible $X$ in $\ell_2(\A)$
for this $\cN^0$.
Consider an arbitrary $\e>0$. Choose a sufficiently large $n$ such that
\begin{equation}
\|Q_n F' - F'\|< \frac{\e}{3}.
\end{equation}
As in the proof of Lemma \ref{lem:directsum_totbu}, we see that
$X':=\{Q_n x_i\}$ is an admissible set for $Q_n \cN^0=\cN^0_n$ 
(by (\ref{eq:coher}))
and
\begin{equation}\label{eq:dliapoekpolunorma1}
d_{X',\F}(Q_n u, Q_n z) =d_{X,\F}(Q_n u,Q_n z). 
\end{equation}
Thus by item (2), we can find 
an $\e/3$-net $\{Q_n F'(y_1),\dots,Q_n F'(y_s)\}$ in $Q_n F'(\M)$
for $d_{X',\F}$. We claim that $\{F'(y_1),\dots, F'(y_s)\}$
is an $\e$-net in $F'(\M)$ for $d_{X,\F}$. Indeed,
for any $y\in B$ 
find a number $k\in \{1,\dots,s\}$ such that
$$
d_{X',\F}(Q_n F'(y_k), Q_n F'(y))<\e/3.
$$
Then by (\ref{eq:dliapoekpolunorma1}) and (\ref{eq:sravnsobych}),
$$
d_{X,\F}(F'(y_k),F'(y))\le 
d_{X,\F}(F'(y_k), Q_n F'(y_k))
+ d_{X,\F}(Q_n F'(y_k), Q_n F'(y))
$$
$$
\hfill +d_{X,\F}(Q_n F'(y),F'(y))
$$
$$
\le \|F'(y_k)- Q_n F'(y_k)\|
+ d_{X',\F}(Q_n F'(y_k), Q_n F'(y))
+\|Q_n F'(y)-F'(y)\| 
$$
$$
<\frac{\e}{3} +\frac{\e}{3}+\frac{\e}{3}=\e.
$$
This completes the proof.
\end{proof}



\end{document}